\documentclass[12pt, a4paper]{amsart}

\usepackage{amsmath, amssymb}

\usepackage[a4paper]{geometry}
\usepackage{enumerate}

\usepackage{graphicx}
\usepackage{wrapfig}

\newtheorem{theorem}{Theorem}[section]
\newtheorem{lemma}[theorem]{Lemma}
\newtheorem{proposition}[theorem]{Proposition}
\newtheorem{corollary}[theorem]{Corollary}
\newtheorem*{maintheorem1}{Theorem A}
\newtheorem*{maintheorem2}{Theorem B}
\newtheorem*{corollaryC}{Corollary C}

\theoremstyle{remark}

\newcommand{\C}{\ensuremath{\mathbb{C}}}
\newcommand{\R}{\ensuremath{\mathbb{R}}}
\newcommand{\SO}{\ensuremath{\mathsf{SO}}}
\newcommand{\SU}{\ensuremath{\mathsf{SU}}}
\newcommand{\U}{\ensuremath{\mathsf{U}}}
\newcommand{\g}[1]{\ensuremath{\mathfrak{#1}}}

\newcommand{\II}{\ensuremath{I\!I}}

\newcommand{\mean}{\ensuremath{\mathcal{H}}}

\DeclareMathOperator{\Ad}{Ad}
\DeclareMathOperator{\ad}{ad}
\DeclareMathOperator{\Exp}{Exp}
\DeclareMathOperator{\Id}{Id}
\DeclareMathOperator{\spann}{span}

\begin{document}

\title[Homogeneous CR submanifolds of complex hyperbolic spaces]
{Homogeneous CR submanifolds\\of complex hyperbolic spaces}

\author[J.~C.~D\'{\i}az-Ramos]{Jos\'{e} Carlos D\'{\i}az-Ramos}
\address{Department of Mathematics, University of Santiago de Compostela, Spain.}
\email{josecarlos.diaz@usc.es}
\author[M.~Dom\'{\i}nguez-V\'{a}zquez]{Miguel Dom\'{\i}nguez-V\'{a}zquez}
\address{Department of Mathematics, University of Santiago de Compostela, Spain.}
\email{miguel.dominguez@usc.es}
\author[O.~P\'erez-Barral]{Olga P\'erez-Barral}
\address{Department of Mathematics, University of Santiago de Compostela, Spain.}
\email{olgaperez.barral@usc.es}

\thanks{The authors have been supported by projects MTM2016-75897-P, PID2019-105138GB-C21 (AEI/FEDER,
Spain) and ED431C 2019/10, ED431F 2020/04 (Xunta de Galicia, Spain). The second author acknowledges support of the
Ram\'{o}n y Cajal program of the Spanish State Research Agency.}



\begin{abstract}
We classify homogeneous CR submanifolds in complex hyperbolic spaces arising as orbits of a subgroup of the solvable part of the Iwasawa decomposition of the isometry group of the ambient space.
\end{abstract}


\maketitle

\section{Introduction}
In the setting of complex analysis, a submanifold of a K\"ahler manifold is said to be CR if the maximal holomorphic subspaces of all tangent spaces define a smooth distribution. Bejancu~\cite{bejancu} introduced a stronger notion of CR submanifold of a K\"ahler manifold by requiring the complementary distribution to the maximal complex distribution in the tangent bundle to be totally real; thus, complex and totally real submanifolds are special examples of CR submanifolds. These two definitions of CR submanifold coincide under the assumption that the complementary distribution has real dimension one. CR submanifolds satisfying this condition are said to be of hypersurface type and they play an important role in the context of complex analysis and boundary value problems. 

In submanifold geometry, an interesting problem is to classify homogeneous CR submanifolds, according to Bejancu's definition, in certain important families of K\"ahler manifolds  such as Hermitian symmetric spaces or, more particularly, complex space forms, that is, complex Euclidean spaces $\mathbb{C}^{n}$, complex projective spaces $\mathbb{C}P^{n}$, and complex hyperbolic spaces $\mathbb{C}H^{n}$. 
The relevance of homogeneous CR submanifolds in this setting stems from the fact that they include several examples of interest in the context of symmetric spaces. 

Real hypersurfaces, that is, submanifolds of real codimension one, constitute an important subclass of CR submanifolds that has been thoroughly studied by many authors. Indeed, the classification of homogeneous real hypersurfaces in $\mathbb{C}^{n}$ follows from Segre's classical work on isoparametric hypersurfaces, whereas for $\mathbb{C}P^{n}$ this is due to Takagi~\cite{takagi}. The classification of homogeneous hypersurfaces in complex hyperbolic spaces is more involved, although it has successfully been solved by Berndt and Tamaru~\cite{BT13}.

Another subclass of homogeneous CR submanifolds is that of homogeneous K\"ahler ones. The corresponding classifications in complex space forms have been achieved by several authors. While in complex $n$-dimensional Euclidean and hyperbolic spaces the only examples are totally geodesic $\mathbb{C}^{k}$ and $\mathbb{C}H^{k}$, with $k<n$, respectively, as proved by Di Scala, Ishi and Loi~\cite{discala}, the classification of compact homogeneous K\"ahler submanifolds of $\C P^n$, obtained by Takeuchi~\cite{takeuchi}, includes more examples.

Lagrangian submanifolds, that is, totally real submanifolds of maximal dimension, constitute another important class of CR submanifolds. The classification of homogeneous Lagrangian submanifolds in complex space forms is still an open problem. However, some partial results have been achieved. 
For example, Bedulli and Gori~\cite{BG} obtained the classification of homogeneous Lagrangian submanifolds in $\C P^n$ induced by the action of a simple compact subgroup of $\SU(n+1)$, whereas little is known in the non-simple case. 
Under additional assumptions, such as the parallelity of the second fundamental form, some results have also been derived; see~\cite{ohnita} for a survey.
However, the classification of homogeneous Lagrangian submanifolds in complex hyperbolic spaces $\C H^n$ has been shown to be a rather complicated problem, mainly due to the non-compactness of its isometry group. Hashinaga and Kajigaya obtained some partial results in~\cite{hashi}. In particular, they derived a classification result of homogeneous Lagrangian submanifolds that arise as orbits of a subgroup of the solvable part of the Iwasawa decomposition of the isometry group of $\mathbb{C}H^{n}$.

Other interesting examples of homogeneous CR submanifolds in Hermitian symmetric spaces arise as principal orbits of polar and coisotropic actions. An isometric action on a Riemannian manifold is said to be polar if there is a submanifold that intersects all the orbits of the action and every such intersection is orthogonal. An isometric action on a Hermitian manifold is called coisotropic if the normal spaces to its principal orbits are totally real. Thus, every principal orbit of a coisotropic action is a CR submanifold. Polar actions on irreducible compact homogeneous K\"{a}hler manifolds are known to be coisotropic~\cite{PT02}, so they produce examples of CR submanifolds. 
In non-compact Hermitian symmetric spaces, deciding whether polar actions are coisotropic is still an open problem. However, this is known to be true in complex hyperbolic spaces: this follows from a classification result in~\cite{DDK}, which also yields several uncountable families of CR submanifolds. Specifically, any orbit of minimum orbit type of any polar action on $\C H^n$ is induced by a subgroup of the solvable part of the Iwasawa decomposition of the isometry group of $\C H^n$. In some cases, such orbits are CR, including the focal sets of homogeneous (or even isoparametric) hypersurfaces with at most three principal curvatures in $\C H^n$~\cite{BD09},~\cite{DDS:advmath}. 
\medskip

The main purpose of this article is to present the classification of homogeneous CR submanifolds in complex hyperbolic spaces that arise as orbits of a subgroup of the solvable part of the Iwasawa decomposition of the isometry group of the ambient space. We briefly explain here the notation that is used in the main theorems of this paper.

Up to finite quotient, the connected component of the identity of the isometry group of $\C H^n$ is the simple Lie group $G=\SU(1,n)$. Let $KAN$ be its Iwasawa decomposition. Here $K\cong\mathsf{U}(n)$ is the isotropy subgroup of $G$ at some point $o\in\C H^n$. 
The solvable Lie group $AN$ acts simply transitively on $\C H^n$.
Let $\g{a}\oplus\g{n}$ be the Lie algebra of $AN$.
It is known that $\g{a}$ is $1$-dimensional, whereas $\g{n}$ is nilpotent and can be decomposed further as $\g{n}=\g{g}_\alpha\oplus\g{g}_{2\alpha}$, where $\g{g}_{2\alpha}$ is the $1$-dimensional center of $\g{n}$.
Moreover, $\g{g}_\alpha$ is isomorphic, as a vector space, to a complex vector space $\C^{n-1}$. In this paper, the symbol $\ominus$ denotes orthogonal complement.
See Section~\ref{sec:preliminaries} for more details and references. 

\begin{maintheorem1}\label{maintheorem1}
An orbit of the action of a connected Lie subgroup of $AN$ is a CR submanifold of $\mathbb{C}H^{n}$ if and only if it is congruent to the orbit $H\cdot g(o)$, where $H$ is the connected Lie subgroup of $AN$ with Lie algebra $\g{h}$ and $g\in AN$, for one of the following cases:
\begin{enumerate}[\rm \quad (i)]
\item $\g{h}=\g{r}$ and $g\in \Exp\bigl((\g{a}\oplus\g{n})\ominus\g{r}\bigr)$; in this case all the orbits of $H$ are totally real submanifolds that constitute a homogeneous subfoliation of a horosphere foliation.\label{th:A:r}
\item $\g{h}=\g{c}\oplus\g{r}\oplus\g{g}_{2\alpha}$ and
    $g\in \Exp\bigl(\g{a}\oplus(\g{g}_\alpha\ominus(\g{c}\oplus\g{r}))\bigr)$;
    in this case all the orbits of $H$ are CR submanifolds that are congruent to each other, and constitute a homogeneous subfoliation of a horosphere foliation.\label{th:A:crz}
\item $\g{h}=\g{a}\oplus\g{r}$ and $g\in\Exp\bigl((\g{g}_{\alpha}\ominus\mathbb{C}\g{r})\oplus\g{g}_{2\alpha}\bigr)$; in this case the CR orbits are totally real equidistant submanifolds to a totally geodesic $\R H^k$ in $\C H^n$, $k\in\{1,\dots, n\}$.\label{th:A:ar}
\item $\g{h}=\g{a}\oplus\g{c}\oplus\g{r}\oplus\g{g}_{2\alpha}$ and $g\in\Exp(J\g{r})$; in this case the CR orbits are the leaves of a homogeneous polar foliation with exactly one minimal leaf (called Berndt-Br\"uck submanifold) on a totally geodesic $\C H^k$ in $\C H^n$, $k\in\{2,\dots, n\}$.\label{th:A:acrz}
\end{enumerate}
Here, $\g{r}$ stands for a totally real subspace of $\g{g}_\alpha$, and $\g{c}$ for a complex subspace of $\g{g}_\alpha$.
\end{maintheorem1}

The next main goal of this paper is to study the congruence classes of these examples.
In the following theorem $\rho\colon\R\to(0,\infty)$ is the analytic function defined by $\rho(t)=(e^t-1)/t$, $t\neq 0$, $\rho(0)=1$.

\begin{maintheorem2}\label{maintheorem2}
Let $H_{1}$ and $H_{2}$ be two connected Lie subgroups of $AN$, and $\g{h}_1$ and $\g{h}_2$ their Lie algebras.
Assume that $H_1$ and $H_2$ act on $\C H^{n}$ in such a way that $H_{1}\cdot g_{1}(o)$ and $H_{2}\cdot g_{2}(o)$ are CR submanifolds, with $g_{1}$, $g_{2}\in AN$, as given by Theorem~A.
Then, $H_{1}\cdot g_{1}(o)$ and $H_{2}\cdot g_{2}(o)$ are congruent if and only if $\g{h}_{1}$ and $\g{h}_{2}$ correspond to the same type in Theorem~A, and according to the type:
\begin{enumerate}[\rm \quad (i)]
\item $g_i=\Exp(b_iB+JT_i+W_i+y_i Z)$, with $b_i$, $y_i\in\R$, $W_i\in\g{g}_\alpha\ominus\C\g{r}$, $T_i\in\g{r}$, $i\in\{1,2\}$, and $\rho(b_2/2)\lVert T_1\rVert=\rho(b_1/2)\lVert T_2\rVert$.\label{th:B:r}
\item In this case all the orbits are congruent.\label{th:B:crz}
\item $g_i=\Exp(W_i+y_i Z)$, with $y_i\in\R$, $W_i\in\g{g}_\alpha\ominus\C\g{r}$, $i\in\{1,2\}$, and $\lVert W_1\rVert=\lVert W_2\rVert$, $\lvert y_1\rvert=\lvert y_2\rvert$.\label{th:B:ar}
\item $g_i=\Exp(JT_i)$, with $T_i\in\g{r}$, $i\in\{1,2\}$, and $\lVert T_1\rVert=\lVert T_2\rVert$.\label{th:B:acrz}
\end{enumerate}
\end{maintheorem2}

%
%

As a consequence we have
\begin{corollaryC}
The moduli space of congruence classes of (non-trivial and proper) homogeneous CR submanifolds of $\C H^n$ induced by the action of subgroups of $AN$ is given by the disjoint union
\[
\bigl(I_{n-1}\times[0,\infty)\bigr)\sqcup I_{0,n-1}\sqcup\bigl([0,\infty)\sqcup(I_{n-1}\times[0,\infty)^2)\bigr)\sqcup\bigl(I_{n-1}\sqcup (I_{1,n-1}\times[0,\infty))\bigr),
\]
where $I_k=\{1,\dots, k\}$ and $I_{k,l}=\{(i,j)\in \mathbb{Z}^2: k\leq i\leq j\leq l\}$. 
\end{corollaryC}
(In the definition of $I_{k,l}$ we think $i=\dim_\R \g{r}$ and $j=\dim_\C(\g{g}_\alpha\ominus\g{c})$.)
\medskip

This article is organized as follows. In Section 2, we introduce the notation and known results that we will use throughout this paper. Section 3 is devoted to classifying homogeneous CR submanifolds in $\C H^n$ which arise as orbits of a connected subgroup of the solvable part of the Iwasawa decomposition of the isometry group of $\mathbb{C}H^{n}$. This problem is tackled in two steps. We first determine the subgroups producing at least a CR orbit (Proposition~\ref{prop:CR-o}). Then, we prove Theorem~A, where we present the classification result. Finally, in Section 4, we study the congruence classes of the examples, and prove Theorem~B.

\section{Preliminaries}\label{sec:preliminaries}
In this section we introduce some known results and notation that we use in this~paper.

\subsection{CR submanifolds}\hfill

Consider a complex Euclidean space $V$ with complex structure $J$ and inner product $\langle\cdot,\cdot\rangle$. A subspace $W\subset V$ is said to be complex if it is invariant by the complex structure, that is, if $JW\subset W$. The subspace $W$ is said to be totally real if $JW$ is perpendicular to $W$. In the setting of Hermitian manifolds, one can generalize these concepts by introducing the notions of complex and totally real submanifolds. Let $\bar{M}$ be a Hermitian manifold with complex structure $J$. A submanifold $M\subset\bar{M}$ is said to be complex (totally real) if at each point $p\in M$ the tangent space $T_{p}M$ is a complex (totally real) vector subspace of $T_{p}\bar{M}$. The subspace $J(T_{p}M)\cap T_{p}M$ is the maximal complex subspace of $T_{p}M$.

The notion of CR submanifold of a Hermitian manifold includes both complex and totally real submanifolds as particular examples. A submanifold $M\subset\bar{M}$ is said to be a CR (Cauchy-Riemann or complex-real) submanifold if there exists a pair of orthogonal complementary distributions of the tangent bundle $TM=\mathfrak{C}\oplus\mathfrak{R}$, where $\mathfrak{C}$ is complex and $\mathfrak{R}$ is totally real. In other words, $M$ is a CR submanifold of $\bar{M}$ if the maximal complex subspaces of the tangent spaces to $M$ have constant dimension along $M$ and their orthogonal complements in each tangent space are totally real subspaces. We refer to~\cite{bejancu} and~\cite{djoric} for more information on CR submanifolds of Hermitian manifolds.

\subsection{The complex hyperbolic space}\hfill

In what follows, we denote by $\mathbb{C}H^{n}$ the complex hyperbolic space of constant holomorphic sectional curvature $-1$. The complex hyperbolic space is known to be a symmetric space of non-compact type and rank one. As a symmetric space, it can be identified with the quotient space $G/K$, where $G=\SU(1,n)$ is, up to a finite quotient, the connected component of the identity element of the isometry group of $\mathbb{C}H^{n}$, and $K=G_{o}=\mathsf{S}(\mathsf{U}(1)\mathsf{U}(n))$ is the stabilizer of an element $o\in\mathbb{C}H^{n}$. Let $\mathfrak{g}$ and $\mathfrak{k}$ denote the Lie algebras of $G$ and $K$, respectively, and consider the Cartan decomposition $\mathfrak{g}=\mathfrak{k}\oplus\mathfrak{p}$ with respect to $o$, where $\mathfrak{p}$ denotes the orthogonal complement to $\mathfrak{k}$ with respect to the Killing form $B$ of $\mathfrak{g}$.
Denote by $\theta$ the associated Cartan involution, which satisfies $\theta|_{\mathfrak{k}}=\Id_{\mathfrak{k}}$ and $\theta|_{\mathfrak{p}}=-\Id_{\mathfrak{p}}$. Denote by $\ad$ and $\Ad$ the adjoint maps of $\mathfrak{g}$ and $G$, respectively. One can define a positive definite inner product $B_\theta$ on  $\mathfrak{g}$ by $B_{\theta}(X,Y):=-B(\theta X,Y)$. This inner product satisfies $B_{\theta}(\ad(X)Y,Z)=-B_{\theta}(Y,\ad(\theta X)Z)$, for all $X$, $Y$, $Z\in\mathfrak{g}$. Moreover, we can identify $\mathfrak{p}\cong T_{o}\mathbb{C}H^{n}$.

We select a maximal abelian subspace $\mathfrak{a}\subset\mathfrak{p}$.
Then $\g{a}$ is $1$-dimensional since $\mathbb{C}H^{n}$ is a rank one symmetric space,
and $\g{a}$ determines a geodesic through $o$.
For each covector $\lambda\in\mathfrak{a}^{*}$, we define the vector subspace $\mathfrak{g}_{\lambda}=
\{X\in\mathfrak{g}:\ad(H)X=\lambda(H)X, \text{ for each $H\in\mathfrak{a}$}\}$. If $\mathfrak{g}_{\lambda}\neq0$, then $\mathfrak{g}_{\lambda}$ is said to be a restricted root space, and each $\lambda\neq0$ such that $\mathfrak{g}_{\lambda}\neq0$ is called a restricted root. Notice that $\mathfrak{g}_{0}$ is always a restricted root space since $\mathfrak{a}\subset\mathfrak{g}_{0}$. It is known that in the case of the complex hyperbolic space the set of restricted roots consists exactly of four elements, $\Sigma=\{\pm\alpha, \pm2\alpha\}$. Then, $\mathfrak{a}$ determines a root space decomposition $\mathfrak{g}=\mathfrak{g}_{-2\alpha}\oplus\mathfrak{g}_{-\alpha}\oplus\mathfrak{g}_{0}
\oplus\mathfrak{g}_{\alpha}\oplus\mathfrak{g}_{2\alpha}$, which is an orthogonal direct sum with respect to $B_\theta$. Moreover, $[\mathfrak{g}_{\lambda},\mathfrak{g}_{\mu}]=\mathfrak{g}_{\lambda+\mu}$ and $\theta\mathfrak{g}_{\lambda}=\mathfrak{g}_{-\lambda}$ for all $\lambda$, $\mu\in\g{a}^*$.

Two unit speed, complete geodesics $\gamma_1$ and $\gamma_2$ in $\C H^n$ are said to be asymptotic if $d(\gamma_1(t),\gamma_2(t))$ remains bounded for large $t$, where $d$ denotes the Riemannian distance function.
This is an equivalence relation.
The ideal boundary of $\C H^n$, denoted by $\C H^n(\infty)$, is the quotient set by this relation.
The union $\overline{\C H^n}=\C H^n\cup\C H^n(\infty)$, when endowed with the cone topology, becomes homeomorphic to the closed unit ball of $\R^n$.

Now we choose a positivity criterion on $\Sigma$ such that $\Sigma^{+}=\{\alpha,2\alpha\}$ is the set of positive roots.
Equivalently, the geodesic determined by~$\g{a}$ has two limit points in the ideal boundary $\C H^n(\infty)$ of $\C H^n$;
choosing a positivity criterion in $\Sigma$ is the same as choosing one of these two limit points at infinity.
We denote by $x\in\C H^n(\infty)$ the point at infinity determined by this choice of positivity.
We define $\mathfrak{n}=\mathfrak{g}_{\alpha}\oplus\mathfrak{g}_{2\alpha}$, which turns out to be a $2$-step nilpotent Lie algebra. The Iwasawa decomposition theorem states that $\mathfrak{g}=\mathfrak{k}\oplus\mathfrak{a}\oplus\mathfrak{n}$ is a direct sum of vector spaces, and that there exists an analytic diffeomorphism $K\times A\times N\to G$, $(k,a,n)\mapsto kan$, where $A$ and $N$ denote the connected Lie subgroups of $G$ with Lie algebras $\mathfrak{a}$ and $\mathfrak{n}$, respectively.
We have $\mathfrak{g}_{0}=\mathfrak{k}_{0}\oplus\mathfrak{a}$, where $\mathfrak{k}_{0}=\mathfrak{g}_{0}\cap\mathfrak{a}\simeq\mathfrak{u}(n-1)$ is the normalizer of $\mathfrak{a}$ in $\mathfrak{k}$.
Both $\mathfrak{g}_{\alpha}$ and $\mathfrak{g}_{2\alpha}$ are normalized by $\mathfrak{k}_{0}$.
In fact, the corresponding connected Lie subgroup $K_0$ acts trivially on $\g{a}$ and $\g{g}_{2\alpha}$, and transitively on the unit sphere of $\g{g}_\alpha$.
It is known that $\mathfrak{a}\oplus\mathfrak{n}$ is a solvable Lie algebra, and that its associated connected Lie subgroup $AN$ acts simply and transitively on $\mathbb{C}H^{n}$. One can endow $AN$, and so $\mathfrak{a}\oplus\mathfrak{n}$, with a left-invariant metric $\langle\,\cdot\,,\,\cdot\,\rangle$ and a complex structure $J$ that make $\mathbb{C}H^{n}$ and $AN$ isomorphic as K\"ahler manifolds. Moreover, up to scaling of $B_\theta$, we have $\langle X,Y\rangle= B_{\theta}(X_{\mathfrak{a}},Y_{\mathfrak{a}})+\frac{1}{2} B_{\theta}(X_{\mathfrak{n}},Y_{\mathfrak{n}})$, for any $X$, $Y\in\mathfrak{a}\oplus\mathfrak{n}$, where the subscripts mean the $\mathfrak{a}$ and $\mathfrak{n}$ components, respectively. The complex structure $J$ on $\mathfrak{a}\oplus\mathfrak{n}$ satisfies that $\mathfrak{g}_{\alpha}$ is a $J$-invariant subspace, and $J\mathfrak{a}=\mathfrak{g}_{2\alpha}$.
The orbits of the action of $N$ on $\C H^n$ are horospheres centered at the point of infinity $x$ chosen above.
In fact, the group $K_0AN$ is a parabolic subgroup determined by $x$, that is, $K_0AN$ is the stabilizer in $\SU(1,n)$ of the point at infinity $x\in\C H^n(\infty)$.

Let $B\in\mathfrak{a}$ be a unit vector and define $Z=JB\in\mathfrak{g}_{2\alpha}$. In particular, $\langle B,B\rangle=B_{\theta}(B,B)=1$ and $2\langle Z,Z\rangle=B_{\theta}(Z,Z)=2$. Moreover, if $U$, $V\in\mathfrak{g}_{\alpha}$, the Lie bracket of $\mathfrak{a}\oplus\mathfrak{n}$ is given by the following relations:
\begin{equation}\label{eq:brackets}
\begin{aligned}%
{}[B,Z]&{}=Z,&  
[B,U]&{}=\frac{1}{2}\,U,&
[U,V]&{}=\langle JU,V\rangle Z,&
[U,Z]&{}=0.&
\end{aligned}
\end{equation}

Using these formulas we get

\begin{lemma}\label{lemma:Ad}
Let $a$, $b$, $x$, $y\in\R$ and $X$, $Y\in\g{g}_\alpha$, and define $g=\Exp(bB+X+yZ)$.~Then,
\begin{align*}
\Ad(g)(aB+Y+xZ)={}&
aB+e^{b/2}Y-\frac{a}{2}\rho\Bigl(\frac{b}{2}\Bigr)X
+\Bigl(xe^b-ay\rho(b)+e^{b/2}\rho\Bigl(\frac{b}{2}\Bigr)\langle JX,Y\rangle\Bigr)Z,
\end{align*}
where $\rho\colon\R\to(0,\infty)$ is the analytic function given by $\rho(t)=(e^t-1)/t$, $t\neq 0$, $\rho(0)=1$.
\end{lemma}

\begin{proof}
Using the bracket relations~\eqref{eq:brackets} it is easy to prove by induction
\begin{align*}
\ad(bB+X+yZ)(aB+Y+xZ)={}&
\frac{b}{2}Y-\frac{a}{2}X+\bigl(bx-ay+\langle JX,Y\rangle\bigr)Z,\\
\ad^{k}(bB+X+yZ)(aB+Y+xZ)={}&
b^{k-1}\Bigl(\frac{b}{2^{k}}Y-\frac{a}{2^k}X
+\bigl(bx-ay+2\Bigl(1\!-\!\frac{1}{2^k}\Bigr)\langle JX,Y\rangle\bigr)Z\Bigr),
\end{align*}
for any $k\geq 2$. Now, recalling that
\[
\Ad(\Exp(X))Y=e^{\ad(X)}Y=\sum_{k=0}^\infty\frac{1}{k!}\ad^k(X)Y,
\]
for any $X$, $Y\in\g{g}$, the result follows after grouping terms and doing some calculations.
\end{proof}

Finally, we recall the expression for the Levi-Civita connection $\bar{\nabla}$ of the complex hyperbolic space (see, for example, \cite{BD09} or \cite{BTV}):
\begin{equation}\label{eq:Levi-Civita}
\begin{aligned}
\bar{\nabla}_{aB+U+xZ}(bB+V+yZ)={}&
\Bigl(\frac{1}{2}\langle U,V\rangle+xy\Bigr)B
-\frac{1}{2}\bigl(bU+yJU+xJV\bigr)\\
&{}+\Bigl(\frac{1}{2}\langle JU,V\rangle-bx\Bigr)Z.
\end{aligned}
\end{equation}

\section{Proof of Theorem~A}

The aim of this section is to find all homogeneous CR submanifolds in complex hyperbolic spaces $\C H^n$ that arise as orbits of connected subgroups of the solvable part of the Iwasawa decomposition of $G=\SU(1,n)$.
Hence, we will determine the connected subgroups $H$ of $AN$ that act on $\mathbb{C}H^{n}$ producing a CR orbit.

Let $H$ be a connected Lie subgroup of $AN$, one of whose orbits is CR.
Since $AN$ acts transitively on $\C H^n$, we may assume that the orbit that is CR is precisely the one through the point $o\in \C H^n$ that determines the compact subgroup $K$ of $G$.
Moreover, since the isometries of $AN$ are holomorphic, a homogeneous submanifold of $\C H^n$ is CR if and only if its tangent space is a CR subspace of the tangent space of $\C H^n$ at some point.
Therefore, it follows that the problem of classifying homogeneous CR submanifolds in the complex hyperbolic space given by the action of a connected Lie subgroup of $AN$ reduces to finding all the Lie subalgebras $\mathfrak{h}$ of $\g{a}\oplus\g{n}$ that can be decomposed into an orthogonal direct sum of a totally real and a complex subspace.

\begin{proposition}\label{prop:CR-o}
Let $H$ be a connected Lie subgroup of $AN$ acting on $\mathbb{C}H^{n}$ in such a way that the orbit $H\cdot o$ through $o$ is a CR-submanifold. Then, its Lie algebra $\g{h}$ is conjugate to $\g{b}\oplus\g{c}\oplus\g{r}\oplus\g{z}$, where $\g{b}$ is a subspace of $\g{a}$, $\g{c}$ is a complex subspace of $\g{g}_\alpha$, $\g{r}$ is totally real subspace of $\g{g}_{\alpha}$, and $\g{z}$ is a subspace of $\g{g}_{2\alpha}$ containing $[\g{c},\g{c}]$.
\end{proposition}

\begin{proof}
We consider the orthogonal projection
$\pi\colon \mathfrak{g}\to\mathfrak{a}\oplus\mathfrak{g}_{2\alpha}$. We have two possibilities:

Case (a): $\pi$ is not surjective. In this case $\pi(\g{h})\neq\g{a}\oplus\g{g}_{2\alpha}$. Hence, there exists a subspace $\g{w}\subset\g{g}_\alpha$, $a$, $x\in\R$, and $X\in\g{g}_\alpha\ominus\g{w}$ such that $\g{h}=\R(aB+X+xZ)\oplus\g{w}$. Here and henceforth $\ominus$ denotes orthogonal complement.

Let $U$, $V\in\g{w}$. Since $\g{h}$ is a Lie subalgebra, $\langle JU,V\rangle Z=[U,V]\in\g{h}\cap\g{g}_{2\alpha}$. Thus $aB+X=0$ or $\langle JU,V\rangle=0$.

Assume first that $aB+X=0$, that is, $a=0$ and $X=0$. Then, $\g{h}=\g{w}\oplus\g{z}$, where $\g{z}=\g{g}_{2\alpha}$ if $x\neq 0$, or $\g{z}=0$ if $x=0$.
We define the maximal complex subspace $\g{c}=\g{w}\cap J\g{w}$ of $\g{w}$ in $\g{g}_\alpha\cong\C^{n-1}$, and $\g{r}=\g{w}\ominus\g{c}$.
Since $J\g{g}_{2\alpha}=\g{a}$, it turns out that $\g{h}\cap J\g{h}=\g{c}$ is also the maximal complex subspace of $\g{h}$. Since $\g{h}$ is CR by assumption, $\g{h}\ominus\g{c}$ is totally real, which implies that $\g{r}$ is totally real.
If $x=0$, that is, if $\g{h}=\g{w}=\g{c}\oplus\g{r}$, we must have $\g{c}=0$, as otherwise, $\g{g}_{2\alpha}=[\g{c},\g{c}]\subset\g{h}$.
Thus, $\g{h}=\g{r}$ is a totally real subspace of $\g{g}_\alpha$, and we take $\g{b}=0$, $\g{z}=0$.
If $x\neq 0$, then $\g{h}=\g{c}\oplus\g{r}\oplus\g{g}_{2\alpha}$, with $\g{c}$ complex in $\g{g}_\alpha$, $\g{r}$ totally real in $\g{g}_\alpha$, and $\g{b}=0$ in the notation of Proposition~\ref{prop:CR-o}.

Therefore, we may assume $aB+X\neq 0$. This implies $\langle JU,V\rangle=0$ for all $U$, $V\in\g{w}$. Hence, $\g{w}$ is totally real as a subspace of $\g{g}_\alpha\cong\C^{n-1}$. Moreover, for each $U\in\g{w}\subset\g{h}$ we have
\[
\frac{a}{2}\,U+\langle JX,U\rangle Z=[aB+X+xZ,U]\in\g{h},
\]
which implies $\langle JX,U\rangle Z\in\g{h}\cap\g{g}_{2\alpha}=0$.
Hence, $X\in\g{g}_\alpha\ominus\C \g{w}$, or equivalently, $\R X\oplus\g{w}\subset\g{g}_\alpha$ is totally real.

If $a\neq 0$, we define $g=\Exp(\frac{2}{a}X+\frac{x}{a}Z)$. Using Lemma~\ref{lemma:Ad} we get
$\Ad(g)(aB+X+xZ)=aB$, and $\Ad(g)(U)=U$ for each $U\in\g{w}$.
Then, $\Ad(g)\g{h}=\g{a}\oplus\g{w}$, where $\g{w}$ is totally real. Thus, $\g{b}=\g{a}$, $\g{c}=0$, $\g{r}=\g{w}$, and $\g{z}=0$ in the notation of Proposition~\ref{prop:CR-o}.

Finally, assume $a=0$.  Then, $X\neq 0$.
In this case we define $g=\frac{x}{\lVert X\rVert^2}JX$. Using Lemma~\ref{lemma:Ad} we get
$\Ad(g)(X+xZ)=X$, and $\Ad(g)(U)=U$ for each $U\in\g{w}$,
that is, $\Ad(g)\g{h}=\R X\oplus\g{w}$.
Thus, we take $\g{b}=0$, $\g{c}=0$, $\g{z}=0$, and $\g{r}=\R X\oplus\g{w}$ is a totally real subspace of $\g{g}_\alpha$.

Case (b): $\pi$ is surjective, that is, $\pi(\g{h})=\mathfrak{a}\oplus\mathfrak{g}_{2\alpha}$.
Then, there exist a subspace $\g{w}\subset\g{g}_{\alpha}$ and $X$, $Y\in\g{g}_{\alpha}\ominus\g{w}$ such that
$\g{h}=\R(B+X)\oplus\g{w}\oplus\R(Y+Z)$.

For any $U\in\g{w}\subset\g{h}$ we have $\frac{1}{2}U+\langle JX,U\rangle Z=[B+X,U]\in\g{h}$.
Thus, $\langle JX,U\rangle Z\in\g{h}\cap\g{g}_{2\alpha}$. Hence $Y=0$ or $\langle JX,U\rangle=0$ for each $U\in\g{w}$.

Assume $Y=0$, that is, $\g{h}=\R(B+X)\oplus\g{w}\oplus\g{g}_{2\alpha}$. We first show that $\g{w}$ is a CR subspace of $\g{g}_\alpha$. 
Let $\g{c}=\g{w}\cap J\g{w}$ be the maximal complex subspace of $\g{w}$. 
Since $B+X$, $Z\in\g{h}\ominus\g{c}$ and $\langle J(B+X),Z\rangle\neq 0$, $\g{h}\ominus\g{c}$ is not totally real. 
Then there exists $\xi'\in(\g{h}\cap J\g{h})\ominus\g{c}$, $\xi'\neq 0$. Let us put $\xi'=a(B+X)+W'+xZ$, for some $W'\in\g{w}$, and where $a$ and $x$ cannot vanish simultaneously. By assumption, $-x(B+X)+xX+aJX+JW'+aZ=-xB+aJX+JW'+aZ=J\xi'\in\g{h}$, and then $xX+aJX+JW'\in\g{w}$. Thus we can take $\xi=(a\xi'-xJ\xi')/(a^2+x^2)\in(\g{h}\cap J\g{h})\ominus\g{c}$, which is of the form $\xi=B+X+W$, with $W\in\g{w}$. Hence, $J\xi=JX+JW+Z\in\g{h}$, where $JX+JW\in\g{w}$. Then $\eta=JX+JW-(\|X\|^2+\|W\|^2)Z\in\g{h}\ominus\C\xi$. Let us decompose $\eta=\eta_c+\eta_r$, where $\eta_c\in\g{h}\cap J\g{h}$ and $\eta_r\in \g{h}\ominus(\g{h}\cap J\g{h})$. Since $\g{h}$ is CR, then $J\eta=J\eta_c+J\eta_r$ with $J\eta_c\in\g{h}$ and $J\eta_r\in(\g{a}\oplus\g{n})\ominus\g{h}$. But
\begin{align*}
J\eta&=-X-W+(\|X\|^2+\|W\|^2)B\\
&=\left(\frac{\|W\|^2}{1+\|X\|^2}(B+X)-W\right)+\frac{1+\|X\|^2+\|W\|^2}{1+\|X\|^2}(\|X\|^2B-X),
\end{align*}
where the first addend belongs to $\g{h}$ and the second one is orthogonal to $\g{h}$. We deduce in particular that $J\eta_c=\frac{\|W\|^2}{1+\|X\|^2}(B+X)-W$, and thus, $-\frac{\|W\|^2}{1+\|X\|^2}(Z+JX)+JW=\eta_c\in\g{h}$. Since $Z$, $JX+JW\in\g{h}$, we get  $JX$, $JW\in\g{h}$. In particular, $JX\in\g{w}$. Then \[
\g{h}=\C(B+X)\oplus\R(\|X\|^2Z-JX)\oplus(\g{w}\ominus\R JX)
\]
is a $\C$-orthogonal direct sum, from where we deduce that $\g{w}$ is a CR subspace of $\g{g}_\alpha$.
Now let $g=\Exp(2X)$. Then, Lemma~\ref{lemma:Ad} yields $\Ad(g)(B+X)=B$, $\Ad(g)(U)=U+2\langle JX,U\rangle Z$ for each $U\in\g{w}$, and $\Ad(g)(Z)=Z$.
This implies $\Ad(g)\g{h}=\g{a}\oplus\g{w}\oplus\g{g}_{2\alpha}$. Since $\g{w}$ is a CR subspace of $\g{g}_\alpha$, we have the orthogonal decomposition $\Ad(g)\g{h}=\g{a}\oplus\g{c}\oplus\g{r}\oplus\g{g}_{2\alpha}$, with $\g{c}$ complex and $\g{r}$ totally real in $\g{g}_\alpha$.


Hence, we assume from now on $Y\neq 0$. Recall that this implies $\langle JX,U\rangle=0$ for each $U\in\g{w}$, that is, $X\in\g{g}\ominus\C\g{w}$.
Similarly, for each $U\in\g{w}$, the fact that $\g{h}$ is a Lie subalgebra yields $\langle JU,Y\rangle Z=[U,Y+Z]\in\g{h}\cap\g{g}_{2\alpha}=0$, which implies $Y\in\g{g}_\alpha\ominus\C\g{w}$.
Moreover, for each $U$, $V\in\g{w}\subset\g{h}$ we have $\langle JU,V\rangle Z=[U,V]\in\g{h}\cap\g{g}_{2\alpha}=0$. Thus, $\g{w}$ is totally real. We also have
\[
\frac{1}{2}\,Y+\bigl(1+\langle JX,Y\rangle\bigr)Z=[B+X,Y+Z]\in\mathfrak{h}.
\]
Since $Y\neq 0$, we get $1+\langle JX,Y\rangle=1/2$, that is, $\langle JX,Y\rangle=-1/2$.

By assumption, $\g{h}$ is a CR subspace of $\g{a}\oplus\g{n}\cong\C^n$. Since $\g{h}=\spann\{B+X,Y+Z\}\oplus\g{w}$ is a $\C$-orthogonal direct sum, it follows that $\R(B+X)\oplus\R(Y+Z)$ is either complex or totally real.
Observe that
\begin{align*}
0&{}=\langle J(B+X),Y+Z\rangle=\langle Z+JX,Y+Z\rangle=1+\langle JX,Y\rangle
\end{align*}
implies $\langle JX,Y\rangle=-1$, which contradicts $\langle JX,Y\rangle=-1/2$. Consequently, $\R(B+X)\oplus\R(Y+Z)$ is a complex subspace. Since $J(B+X)=Z+JX$, then $Y=JX$. Hence,
$-1/2=\langle JX,Y\rangle=\lVert X\rVert^{2}$, which gives a contradiction. Thus, this case is not possible.
\end{proof}

Now that we know the subgroups $H$ of $AN$ that have a CR orbit through $o\in\C H^n$, we must study which orbits of the $H$-action are CR submanifolds.
Since $AN$ acts transitively on the complex hyperbolic space, it will be enough to determine the elements $g\in AN$ such that the orbit $H\cdot g(o)$ is a CR submanifold. The next result reduces the set of elements of $AN$ to investigate.

\begin{lemma}\label{lemma:slice}
Let $\g{h}=\g{b}\oplus\g{c}\oplus\g{r}\oplus\g{z}$, with $\g{b}$ a subspace of $\g{a}$, $\g{c}$ complex in $\g{g}_\alpha$, $\g{r}$ totally real in $\g{g}_\alpha$, and $\g{z}$ a subspace of $\g{g}_{2\alpha}$ such that $[\g{c},\g{c}]\subset\g{z}$.
Let $H$ be the connected Lie subgroup of $AN$ whose Lie algebra is $\g{h}$.
Then, each orbit of $H$ can be written as $H\cdot \Exp(X)(o)$ with $X\in(\g{a}\oplus\g{n})\ominus\g{h}$.
\end{lemma}

\begin{proof}
The Lie algebra $\g{h}$ can be identified with the tangent space to the orbit $H\cdot o$ at $o$, and then, the corresponding normal space can be identified with the orthogonal complement of $\g{h}$ in $\g{a}\oplus\g{n}$, $\nu_{o}(H\cdot o)=(\g{a}\oplus\g{n})\ominus\g{h}=
(\g{a}\ominus\g{b})\oplus\g{c}'\oplus J\g{r}\oplus(\g{g}_{2\alpha}\ominus\g{z})$,
where $\g{c}'=\g{g}_\alpha\ominus(\g{c}\oplus\g{r}\oplus J\g{r})$ is a complex subspace of $\g{g}_{\alpha}$.
We denote $\Sigma=\Exp\bigl((\g{a}\ominus\g{b})\oplus\g{c}'\oplus J\g{r}\oplus(\g{g}_{2\alpha}\ominus\g{z})\bigr)(o)$,
which is a submanifold of $AN$ since $\Exp\colon\g{a}\oplus\g{n}\to AN$ is a diffeomorphism (but not a subgroup in general). We show that $\Sigma$ intersects every orbit of the $H$-action.
In fact, it is enough to show that the smooth map
$\varphi\colon H\times\Sigma\to AN$, $(h,p)\mapsto hp$,
is surjective.

Let $g\in AN$. Since $\Exp\colon\g{a}\oplus\g{n}\to AN$ is a diffeomorphism, there exist $c$, $z\in\R$, and $W\in\g{g}_\alpha$ such that $g=\Exp(cB+W+zZ)$.
Since $\g{g}_\alpha=\g{c}\oplus\g{c}'\oplus\g{r}\oplus J\g{r}$, we write
$W=U+V+S+JT$, with $U\in\g{c}$, $V\in\g{c}'$, and $S$, $T\in\g{r}$, accordingly.

If $\g{b}=0$ we set $a=0$, $b=c$, and if $\g{b}=\g{a}$ we set $a=c$, $b=0$.
If $\g{z}=0$, we set
\begin{align*}
x&{}=0,&
y&{}=e^{-a}\rho(b)^{-1}\Bigl(\rho(c)z-\frac{1}{2}\rho\Bigl(\frac{c}{2}\Bigr)^2\langle S,T\rangle\Bigr);
\end{align*}
otherwise, if $\g{z}=\g{g}_{2\alpha}$ we set
\begin{align*}
x&{}=\rho(a)^{-1}\Bigl(\rho(c)z-\frac{1}{2}\rho\Bigl(\frac{c}{2}\Bigr)^2\langle S,T\rangle\Bigr),&
y&{}=0.
\end{align*}

Taking this into account, we define
\begin{align*}
&X=aB+\rho\Bigl(\frac{c}{2}\Bigr)\rho\Bigl(\frac{a}{2}\Bigr)^{-1}\bigl(U+S\bigr)+xZ
\in\g{h},\\
&Y=bB+e^{-a/2}\rho\Bigl(\frac{c}{2}\Bigr)\rho\Bigl(\frac{b}{2}\Bigr)^{-1}\bigl(V+JT\bigr)+yZ
\in(\g{a}\oplus\g{n})\ominus\g{h}.
\end{align*}

Using~\cite[Subsections 4.1.3 and 4.1.4]{BTV} yields
\begin{align*}
\Exp(X)\cdot\Exp(Y)
={}&\Bigl(a,\Exp_{\g{n}}\Bigl(\rho\Big(\frac{c}{2}\Bigr)\bigl(U+S\bigr)+\rho(a)xZ\Bigr)\Bigr)\\%
&{}\cdot\Bigl(b,
\Exp_{\g{n}}\Bigl(e^{-a/2}\rho\Big(\frac{c}{2}\Bigr)\bigl(V+JT\bigr)+\rho(b)yZ\Bigr)\Bigr)\\%
{}={}&\Bigl(a+b,\Exp_{\g{n}}\Bigl(
\rho\Big(\frac{c}{2}\Bigr)\bigl(U+S+V+JT\bigr)\\%
&\phantom{\Bigl(a+b,\Exp_{\g{n}}}
+\Bigl(\rho(a)x+e^a\rho(b)y
+\frac{1}{2}\rho\Bigl(\frac{c}{2}\Bigr)^2\langle J(U+S),V+JT\rangle)Z\Bigr)\Bigr)\\%
{}={}&\Bigl(c,\Exp_{\g{n}}\Bigl(
\rho\Big(\frac{c}{2}\Bigr)\bigl(U+S+V+JT\bigr)+\rho(c)z\Bigr)\Bigr)\\[1ex]%
{}={}&\Exp\bigl(cB+U+V+S+JT+zZ\bigr)=g,
\end{align*}
which shows that $\varphi$ is onto, as we wanted to prove.
\end{proof}

We can now prove the first main theorem of this paper.

\begin{proof}[Proof of Theorem~A]
Since $H\cdot g(o)$ is homogeneous and $H\subset AN$ acts by holomorphic isometries, the orbit $H\cdot g(o)$ is a CR submanifold of $\mathbb{C}H^{n}$ if and only if the tangent space $T_{g(o)}(H\cdot g(o))$ is a CR subspace of $T_{g(o)}\C H^n$. Since $H\cdot g(o)=g(g^{-1}Hg\cdot o)$, the tangent space to the orbit $H\cdot g(o)$ at $g(o)$ can be written in terms of the Lie algebra $\g{h}$ as $T_{g(o)}\bigl(H\cdot g(o)\bigr)=g_{*}\Ad(g^{-1})\g{h}$.
Since $g$ is a holomorphic isometry, it is enough to study the elements $g\in AN$ such that $\Ad(g)\g{h}$ is a CR subspace of $\g{a}\oplus\g{n}$.
By Lemma~\ref{lemma:slice} we only have to consider elements of the form $g\in \Exp\bigl((\g{a}\oplus\g{n})\ominus\g{h}\bigr)$.

By Proposition~\ref{prop:CR-o}, the Lie subalgebras $\g{h}$ we have to work with are:
\[
\g{h}\in\{\g{r}, \ \g{c}\oplus\g{r}\oplus\g{g}_{2\alpha},\  \g{a}\oplus\g{r},\  \g{a}\oplus\g{c}\oplus\g{r}\oplus\g{g}_{2\alpha}\},
\]
where $\g{r}$ is a totally real subspace of $\g{g}_{\alpha}$, and $\g{c}$ is a complex one.

(i): $\g{h}=\g{r}$, with $\g{r}$ a totally real subspace of $\g{g}_\alpha$. For $g=\Exp(bB+JT+W+yZ)\in AN$ with $b$, $y\in\R$ and $T\in\g{r}$, $W\in\g{g}_\alpha\ominus\C\g{r}$, and any $S\in\g{r}$, Lemma~\ref{lemma:Ad} yields
$\Ad(g)(S)=e^{b/2}S-e^{b/2}\rho(b/2)\langle T,S\rangle Z$.
Hence, we have
\begin{equation}\label{eq:Adg-r}
\Ad(g)\g{h}=(\g{r}\ominus\R T)\oplus\R\Bigl(T-\rho\Bigl(\frac{b}{2}\Bigr)\lVert T\rVert^2 Z\Bigr).
\end{equation}
Since $\g{r}$ is totally real and $J\g{g}_{2\alpha}=\g{a}$,
this readily implies that $\Ad(g)\g{h}$ is totally real.
Therefore all the orbits of $H$ are totally real, and since $\g{h}\subset \g{n}$, each $H$-orbit is contained in one of the leaves of the horosphere foliation induced by the Lie group $N$, from where~(\ref{th:A:r}) of Theorem~A follows.

(ii): $\g{h}=\g{c}\oplus\g{r}\oplus\g{g}_{2\alpha}$, where $\g{c}$ is complex and $\g{r}$ is totally real in $\g{g}_\alpha$. Taking $g=\Exp(bB+X+yZ)$ with $b$, $y\in\R$ and $X\in\g{g}_\alpha\ominus(\g{c}\oplus\g{r})$, and $U\in\g{c}\oplus\g{r}$ we get
\[
\Ad(g)(U+xZ)=e^{b/2}U+\Bigl(xe^b+e^{b/2}\rho\Bigl(\frac{b}{2}\Bigr)\langle JX,U\rangle\Bigr) Z.
\]
Note that $Z\in\Ad(g)(\g{h})$ (just set $U=0$, $x=e^{-b}$).
Hence, it follows that $\g{c}\oplus\g{r}\subset\Ad(g)\g{h}$, and for dimension reasons, $\Ad(g)\g{h}=\g{h}$. Thus, all the orbits of $H$ are CR-submanifolds that are congruent to each other, and since $\g{h}\subset \g{n}$, $H$-orbits are contained in the leaves of the horosphere foliation induced by $N$.
This corresponds to~(\ref{th:A:crz}) of Theorem~A.

(iii): $\g{h}=\g{a}\oplus\g{r}$, with $\g{r}$ a totally real subspace of $\g{g}_\alpha$.
We have $(\g{a}\oplus\g{n})\ominus\g{h}=J\g{r}\oplus(\g{g}_\alpha\ominus\C\g{r})\oplus\g{g}_{2\alpha}$.
Consider $g=\Exp(2JT+2W+yZ)$, where $T\in \g{r}$, $W\in\g{g}_\alpha\ominus\mathbb{C}\g{r}$, and $y\in\mathbb{R}$. For any $a\in\R$ and $S\in\g{r}$, Lemma~\ref{lemma:Ad} yields
\[
\Ad(g)(aB+S)=aB+S-a(JT+W)-(ay+2\langle T,S\rangle)Z.
\]
Then
\begin{equation}\label{eq:Adg-ar}
\Ad(g)\g{h}=\R(B-JT-W-yZ)
\oplus\R(T-2\lVert T\rVert^2 Z)\oplus(\g{r}\ominus\R T).
\end{equation}

If $T=0$, we get $\Ad(g)\g{h}=\R(B-W-yZ)\oplus\g{r}$, which is totally real since $W\in\g{g}_\alpha\ominus\C\g{r}$; in particular, $\Ad(g)\g{h}$ is CR in $\g{a}\oplus\g{n}$.

Assume $T\neq 0$.
Then, $\R\left(T-2\lVert T\rVert^2 Z\right)\oplus(\g{r}\ominus\R T)$ is totally real, and
$B-JT-W-yZ$ is complex orthogonal to $\g{r}\ominus\R T$.
On the other hand,
\[
\langle J(B-JT-W-yZ),T-2\lVert T\rVert^{2}Z \rangle
=-\lVert T\rVert^2\neq 0,
\]
which implies that $\Ad(g)\g{h}$ is not totally real.
Moreover, $J(B-JT-W-yZ)=yB+T-JW+Z$ cannot be proportional to $T-2\lVert T\rVert^2 Z$.
Hence, $\Ad(g)\g{h}$ does not contain a non-trivial complex vector subspace.
Therefore, if $T\neq 0$, $\Ad(g)\g{h}$ is not a CR subspace of $\g{a}\oplus\g{n}$.
We conclude that $\Ad(g)\g{h}$ is CR if and only if $g\in\Exp\bigl((\g{g}_\alpha\ominus\C\g{r})\oplus\g{g}_{2\alpha}\bigr)$,  and in this case $\Ad(g)\g{h}$ is actually totally real. Note also that $H\cdot o$ is a totally geodesic $\R H^k$, with $k=\dim \g{r}+1$, and hence the other $H$-orbits are equidistant to it. This proves item~(\ref{th:A:ar}) of Theorem~A.

(iv): $\g{h}=\g{a}\oplus\g{c}\oplus\g{r}\oplus\g{g}_{2\alpha}$, where $\g{c}$ is complex and $\g{r}$ is totally real in $\g{g}_\alpha$.
In view of Lemma~\ref{lemma:slice} we consider $(\g{a}\oplus\g{n})\ominus\g{h}=\g{c}'\oplus J\g{r}$, where $\g{c}'=\g{g}_\alpha\ominus(\g{c}\oplus\C\g{r})$ is a complex subspace of $\g{g}_\alpha$.
Let $g=\Exp(2JT+2W)$, with $T\in\g{r}$ and $W\in\g{c'}$.
Let $a$, $x\in\R$, $S\in\g{r}$, and $U\in\g{c}$. Using Lemma~\ref{lemma:Ad} we get
\[
\Ad(g)(aB+U+S+xZ)=aB+U+S-a(JT+W)+\bigl(x-2\langle S,T\rangle\bigr)Z.
\]
In particular, $Z=\Ad(g)(Z)\in\Ad(g)\g{h}$, and thus,
\begin{equation}\label{eq:Adg-acrz}
\Ad(g)\g{h}=\mathbb{R}(B-JT-W)\oplus\g{c}\oplus\g{r}\oplus\g{g}_{2\alpha}.
\end{equation}

If $W=0$, the maximal complex distribution of $\Ad(g)\g{h}$ is
\[
\g{m}=\Ad(g)\g{h}\cap J(\Ad(g)\g{h})=
\R(B-JT)\oplus\R(Z+T)\oplus\g{c}.
\]
Then, its orthogonal complement
\[	
\Ad(g)(\g{h})\ominus\g{m}=
(\g{r}\oplus\g{g}_{2\alpha})\ominus\R(Z+T)
=(\g{r}\ominus\R T)\oplus\R\bigl(T-\lVert T\rVert^{2} Z\bigr),
\]
is totally real. Hence $\Ad(g)\g{h}$ is a CR subspace.

Now assume $W\neq 0$.
In this case, the maximal complex subspace of $\Ad(g)\g{h}$ is
$\g{m}=\Ad(g)\g{h}\cap J(\Ad(g)\g{h})=\g{c}$, and its orthogonal complement in $\Ad(g)(\g{h})$ is $\Ad(g)(\g{h})\ominus\g{m}=\mathbb{R}(B-JT-W)\oplus\g{r}\oplus\g{g}_{2\alpha}$, which is not a totally real subspace since
$\langle J(B-JT-W),Z\rangle=1\neq 0$.
Then, $\Ad(g)\g{h}$ is not a CR submanifold when $W\neq 0$.
Altogether this proves that $\Ad(g)\g{h}$ is a CR subspace of $\g{a}\oplus\g{n}$ precisely when $g\in\Exp(J\g{r})$. Note that $\Exp (\g{h}\oplus J\g{r})\cdot o$ is a totally geodesic $\C H^k$, with $k=\dim_\C (\g{h}\oplus J\g{r})$. Then, it follows from~\cite[Theorem~A and Corollary~6.2]{DDK} that the $H$-orbits that foliate this $\C H^k$ constitute a homogeneous polar regular foliation with exactly one minimal leaf. This minimal leaf is precisely $H\cdot o$, which is called a Berndt-Br\"uck submanifold $W^{2k-\dim\g{r}}$ with totally real normal bundle in such $\C H^k$, see~\cite{BB:crelle}, \cite{BD09}. This proves~(\ref{th:A:acrz}).
\end{proof}

\section{Proof of Theorem~B}
This section is devoted to determining the congruence classes of the homogeneous CR submanifolds obtained in Theorem~A.
We first study each case separately.

\subsection*{Case~(\ref{th:A:r})}

Let $\g{h}=\g{r}$ be a totally real subspace of $\g{g}_{\alpha}$.

First of all, recall that two totally real subspaces of $\g{g}_\alpha$ are congruent by an element of $K_0\cong\mathsf{S}(\mathsf{U}(1)\mathsf{U}(n))$ if and only if both have the same dimension.
Hence, we can fix $\g{r}$ in the rest of the proof.

Let $T\in\g{r}$, $W\in\g{g}_\alpha\ominus\C\g{r}$, and $b$, $y\in\R$.
It readily follows from~\eqref{eq:Adg-r} that
\begin{equation}\label{eq:rhoJT}
\Ad\bigl(\Exp(bB+JT+W+yZ)\bigr)\g{h}
=\Ad\Bigl(\Exp\Bigl(\rho\Bigr(\frac{b}{2}\Bigr)JT\Bigr)\Bigr)\g{h}.
\end{equation}
Hence, $H\cdot\Exp(bB+JT+W+yZ)(o)$ and $H\cdot\Exp(\rho(b/2)JT)(o)$ are congruent.
Thus, in order to settle the congruence problem for case~(\ref{th:A:r}) we just have to consider elements $g\in\Exp(J\g{r})$.

\begin{lemma}\label{lemma:shape:Adg-r}
The squared norm of the mean curvature vector  $\mathcal{H}$ of the orbit $H\cdot\Exp(JT)(o)$, $T\in\g{r}$, is given by
\[
\lVert\mean\rVert^{2}=
\frac{4\lVert T\rVert^{2}+(r+(r+1)\lVert T\rVert^{2})^{2}}{4(1+\lVert T\rVert^{2})^{2}},
\]
where $r=\dim\g{r}=\dim\g{h}$.
\end{lemma}

\begin{proof}
Let $g=\Exp(JT)$. Since $H\cdot g(o)$ is congruent to $g^{-1}Hg\cdot o$, we calculate the mean curvature of the latter.  It suffices to do so at $o$ by homogeneity.
It follows from~\eqref{eq:Adg-r} that the normal space of $g^{-1}Hg\cdot o$ at $o$, which can be identified with the orthogonal complement of $\Ad(g^{-1})\g{h}=\Ad(\Exp(-JT))\g{h}$ in $\g{a}\oplus\g{n}$, is
\begin{equation}\label{eq:Adg-r-normal}
\nu_{o}(g^{-1}Hg\cdot o)
=\g{a}\oplus J\g{r}\oplus(\g{g}_\alpha\ominus\C\g{r})\oplus\mathbb{R}(-T+Z).
\end{equation}

Let $S\in\g{r}\ominus\R T$ with $\lVert S\rVert=1$, and
$X=\frac{T+\lVert T\rVert^{2}Z}{\lVert T\rVert\sqrt{1+\lVert T\rVert^{2}}}$ if $T\neq 0$.
Using the formula for the Levi-Civita connection of left-invariant vector fields~\eqref{eq:Levi-Civita}, it follows that
\begin{equation}\label{eq:Adg-r-nabla}
\begin{aligned}
\bar{\nabla}_{S}S
&{}=\frac{1}{2}B, &
\bar{\nabla}_{X}X
&{}=\frac{1+2\lVert T\rVert^{2}}{2(1+\lVert T\rVert^{2})}B-\frac{1}{1+\lVert T\rVert^{2}}JT.
\end{aligned}
\end{equation}

Recall that, given an orthonormal basis $\{E_{i}\}$ of $T_{o}(g^{-1}H g\cdot o)$, the mean curvature can be computed as $\mean=\sum_{i}\II(E_{i},E_{i})$,
where $\II$ denotes the second fundamental form.
In this case, using~\eqref{eq:Adg-r-nabla} and projecting onto the normal space according to~\eqref{eq:Adg-r-normal}, the mean curvature of $g^{-1}Hg \cdot o$ is given by
\[
\mean
=\Bigl(\frac{r-1}{2}+\frac{1+2\lVert T\rVert^{2}}{2(1+\lVert T\rVert^{2})}\Bigr)B
-\frac{1}{1+\lVert T\rVert^{2}}JT.
\]
The result follows after calculating the squared norm of this vector.
\end{proof}

In order to finish the proof in this case, let $g_{1}=\Exp(JT_{1})$, $g_{2}=\Exp(JT_{2})$ with $T_{1}$, $T_{2}\in \g{r}$.
We investigate whether the orbits $H\cdot g_{1}(o)$ and $H\cdot g_{2}(o)$ are congruent.

Assume first $\lVert T_{1}\rVert=\lVert T_{2}\rVert$.
Since the connected component of the identity of the normalizer of $\g{r}$ in $K_0$, which is given by $N_{K_{0}}^{0}(\g{r})\cong \SO(\dim\g{r})\times \U(n-1-\dim\g{r})$, acts transitively on the spheres of $\g{r}$ centered at the origin, there exists an element $k\in N_{K_{0}}^{0}(\g{r})$ satisfying $\Ad(k)(T_{1})=T_{2}$.
Since $k\in N_{K_{0}}^{0}(\g{r})$ and $K_{0}\cong \U(n-1)$,
then $k\in N_{K_{0}}(\g{g}_\alpha\ominus \mathbb{C}\g{r})$.
Considering these facts, it follows that
\begin{align*}
\Ad(k)\Ad(g_{1}^{-1})\g{h}
&{}=\Ad(k)\bigl((\g{r}\ominus\R T_{1})\oplus\R(T_{1}+\lVert T_{1}\rVert^{2}Z)\bigr)\\ &{}=(\g{r}\ominus\R T_{2})\ominus\R (T_{2}+\lVert T_{2}\rVert^{2}Z)
=\Ad(g_{2}^{-1})\g{h}.
\end{align*}
Since $k$ fixes $o\in\mathbb{C}H^{n}$ and normalizes $\g{a}\oplus\g{n}$, it follows that $k( g_{1}^{-1} H g_{1}\cdot o)=g_{2}^{-1} H g_{2}\cdot o$.
This shows that $H\cdot g_{1}(o)$ is congruent to $H\cdot g_{2}(o)$.

Conversely, in view of Lemma~\ref{lemma:shape:Adg-r} it is enough to show that $h\colon[0,\infty)\to[0,\infty)$, given by ${t\mapsto \frac{4t+(r+(r+1)t)^{2}}{4(1+t)^2}}$, is injective.
This follows simply from $h'(t)=\frac{2+r+(r-1)t}{2(1+t)^3}>0$, which implies that $h$ is strictly increasing.
Therefore, if $\lVert T_1\rVert\neq\lVert T_2\rVert$, $H\cdot g_1(o)$ and $H\cdot g_2(o)$ are not congruent.

All in all, and taking \eqref{eq:rhoJT} into account, this means that the orbit $H\cdot\Exp(b_1 B+JT_1+W_1+y_1 Z)(o)$ is congruent to $H\cdot\Exp(b_2 B+JT_2+W_2+y_2 Z)(o)$, with $b_i$, $y_i\in\R$, $T_i\in\g{r}$, $W_i\in\g{g}_\alpha\ominus\C\g{r}$, $i\in\{1,2\}$, if and only if $\rho(b_2/2)\lVert T_1\rVert=\rho(b_1/2)\lVert T_2\rVert$.
This concludes the proof of case~(\ref{th:B:r}) of Theorem~B.

%

\subsection*{Case~(\ref{th:A:crz})}

Let $\g{h}=\g{c}\oplus\g{r}\oplus\g{g}_{2\alpha}$, with $\g{c}$ complex and $\g{r}$ totally real in $\g{g}_\alpha$.

It follows from Theorem~A(\ref{th:A:crz}) that all the orbits of $H$ are congruent to each other.
Let now $H_{1}$ and $H_{2}$ be connected Lie subgroups of $G$ with Lie algebras $\g{h}_{i}=\g{c}_{i}\oplus\g{r}_{i}\oplus\g{g}_{2\alpha}$, where $\g{c}_i$ is complex and $\g{r}_i$ is totally real in $\g{g}_\alpha$, $i\in\{1,2\}$.
Then, since isometries of $\SU(1,n)$ are holomorphic, it follows that $\g{h}_1$ and $\g{h}_2$ are conjugate if and only if $\dim\g{c}_1=\dim\g{c}_2$ and $\dim\g{r}_1=\dim\g{r}_2$.
Hence, the orbits of $H_1$ and $H_2$ are congruent if and only if $\dim\g{c}_1=\dim\g{c}_2$ and $\dim\g{r}_1=\dim\g{r}_2$.

\begin{lemma}\label{lemma:shape:Adg-crz}
The squared norm of the mean curvature of any orbit of $H$ is
\[
\lVert\mean\rVert^2=\frac{\bigl(2+\dim(\g{c}\oplus\g{r})\bigr)^2}{4}.
\]
\end{lemma}

\begin{proof}
From~\cite[Corollary 6.2]{DDK}, we have $2\mean=\bigl(2+\dim(\g{c}\oplus\g{r})\bigr)B$, and the result follows taking squared norm.
\end{proof}

\subsection*{Case~(\ref{th:A:ar})}

Let $\g{h}=\g{a}\oplus\g{r}$, where $\g{r}$ is a totally real subspace of $\g{g}_{\alpha}$.

Since two totally real subspaces of $\g{g}_\alpha$ are conjugate if and only if they have the same dimension, we can fix $\g{r}$ from now on.
From Theorem~A we just have to consider orbits of the form $H\cdot g(o)$, with $g\in\Exp\bigl((\g{g}_\alpha\ominus\C\g{r})\oplus\g{g}_{2\alpha}\bigr)$.
We define $r=\dim\g{r}$.

\begin{lemma}\label{lemma:shape:Adg-ar}
The squared norms of the mean curvature vector and of the second fundamental form of
$H\cdot\Exp(2W+yZ)(o)$, $W\in\g{g}_\alpha\ominus\C\g{r}$, $y\in\R$, are given by
\begin{align*}
\lVert\mean\rVert^{2}
&{}=\frac{(1+r)^2\lVert W\rVert^4+(2+r)^2 y^2(1+y^2)
+\lVert W\rVert^2\bigl(1+8y^2+r^2(1+2y^2)+2r(1+3y^2)\bigr)}
{4(1+y^2+\lVert W\rVert^2)^2},\\
\lVert\II\rVert^2
&{}=\frac{(1+r)\lVert W\rVert^4+(4+3r)y^2(1+y^2)+\lVert W\rVert^2\bigl(1+r+4y^2(2+r)\bigr)}
{4(1+y^2+\lVert W\rVert^2)^2}.
\end{align*}
\end{lemma}

\begin{proof}
Let $g=\Exp(2W+yZ)$.
Recall from~\eqref{eq:Adg-ar} (with $T=0$) that $T_{g(o)}\bigl(H\cdot g(o)\bigr)$ is identified with $\Ad(g^{-1})\g{h}=\Ad(\Exp(-2W-yZ))\g{h}=\R(B+W+yZ)\oplus\g{r}$. We define
\begin{equation}\label{eq:X}
\begin{aligned}
X
&{}=\frac{B+W+yZ}{\sqrt{1+y^2+\lVert W\rVert^2}},&
\xi_1
&{}=\frac{-yB+Z}{\sqrt{1+y^2}},&
\xi_2
&{}=\frac{\lVert W\rVert^2 B-(1+y^2)W+y\lVert W\rVert^2 Z}
{\lVert W\rVert\sqrt{(1+y^2)(1+y^2+\lVert W\rVert^2)}}.
\end{aligned}
\end{equation}
Then, $X$, $\xi_1$, and $\xi_2$ (if $W\neq 0$) are mutually orthogonal unit vectors of $\g{a}\oplus\g{n}$.
Furthermore, we have $\Ad(g^{-1})\g{h}=\R X\oplus\g{r}$, and the normal space of $H\cdot g(o)$ can be identified with the direct sum
$\nu_{g(o)}\bigl(H\cdot g(o)\bigr)=
\R\xi_1\oplus\R\xi_2\oplus J\g{r}\oplus(\g{g}_\alpha\ominus\bigl(\C\g{r}\oplus\R W)\bigr)$.

Let $S$, $T\in\g{r}$.
Using the formula for the Levi-Civita connection~\eqref{eq:Levi-Civita} for left-invariant vector fields of $AN$,
and taking the orthogonal projection onto $\nu_{g(o)}\bigl(H\cdot g(o)\bigr)$ we get
\begin{align*}
\II(X,X)
&{}=(\bar{\nabla}_X X)^\perp
=\frac{1}{1+y^2+\lVert W\rVert^2}
\Bigl(\Bigl(y^2+\frac{1}{2}\lVert W\rVert^2\Bigr)B-\frac{1}{2}W-yZ-yJW\Bigr),\\[1ex]
\II(X,S)
&{}=(\bar{\nabla}_S X)^\perp
=-\frac{y}{2\sqrt{1+y^2+\lVert W\rVert^2}}JS,\\
\II(S,T)
&{}=\frac{\langle S,T\rangle}{2}\bigl(\langle B,\xi_1\rangle\xi_1+\langle B,\xi_2\rangle\xi_2\bigr)
=\frac{\langle S,T\rangle}{2(1+y^2+\lVert W\rVert^2)}\bigl((y^2+\lVert W\rVert^2)B-W-yZ\bigr).
\end{align*}

The squared norms of the mean curvature and of the second fundamental form are calculated as $\lVert\mean\rVert^2=\sum_{i,j}\langle\II(E_i,E_i),\II(E_j,E_j)\rangle$ and $\lVert\II\rVert^2=\sum_{i,j}\lVert\II(E_i,E_j)\rVert^2$ with respect to an orthonormal basis $\{E_i\}$ of the tangent space.
The result follows after substitution and some calculations.
\end{proof}

\begin{corollary}\label{cor:invariants}
We have
\begin{align*}
\lVert\mean\rVert^2-\lVert\II\rVert^2
&{}=\frac{r(1+r)(y^2+\lVert W\rVert^2)}{4(1+y^2+\lVert W\rVert^2)},\\
(r+1)\lVert\II\rVert^2-\lVert\mean\rVert^2
&{}=\frac{ry^2\bigl((3+2r)(1+y^2)+2(3+r)\lVert W\rVert^2\bigr)}{4(1+y^2+\lVert W\rVert^2)^2}.
\end{align*}
\end{corollary}

Let $g_{1}=\Exp(2W_{1}+y_{1}Z)$, $g_{2}=\Exp(2W_{2}+y_{2}Z)$, with $W_{i}\in\g{g}_\alpha\ominus\C\g{r}$ and $y_{i}\in\R$, $i\in\{1,2\}$.
We show that, if $H\cdot g_1(o)$ is congruent to $H\cdot g_2(o)$, then $\lVert W_1\rVert=\lVert W_2\rVert$ and $\lvert y_1\rvert=\lvert y_2\rvert$.

In fact, if $r\geq 1$, taking into account Corollary~\ref{cor:invariants}, the previous claim follows from

\begin{lemma}
The function $F\colon[0,+\infty)\times[0,+\infty)\to [0,+\infty)\times [0,+\infty)$ defined by
\[
F(z,w)=\left(\frac{z+w}{1+z+w},\frac{z\bigl(a(1+z)+(a+3)w\bigr)}{(1+z+w)^2}\right),
\]
where $a\geq 5$, is injective.
\end{lemma}

\begin{proof}
Let $(c_1,c_2)\in[0,+\infty)\times[0,+\infty)$. We have to check whether $F^{-1}(c_1,c_2)$ has at most one element.
In fact, there are two solutions to the equation $F(z,w)=(c_1,c_2)$ which are
\[
(z,w)=\Bigl(\frac{a+3c_1\pm\sqrt{(a+3c_1)^2-12c_2}}{6(1-c_1)}, \frac{3c_1-a\mp\sqrt{(a+3c_1)^2-12c_2}}{6(1-c_1)}\Bigr).
\]
Observe that we need $0\leq c_1<1$ for the first component to be non-negative, whereas $(a+3c_1)^2-12c_2\geq 0$ so that there are real solutions.
Since $3c_1-a\leq -2<0$, the first possibility would give a negative solution for $w$, which is not allowed.
Then $F^{-1}(c_1,c_2)$ has at most one element, and $F$ is injective.
\end{proof}

Now we assume $r=0$, that is, $\g{h}=\g{a}$.
Thus, we have to study the congruence classes of orbits of the $1$-dimensional Lie group $A$ appearing in the Iwasawa decomposition of~$\SU(1,n)$.

Recall that $A\cdot o$ is a geodesic.
Let $\gamma\colon\R\to\C H^n$ be a unit speed parametrization of $A\cdot o$, and assume $\lim_{t\to\infty}\gamma(t)=x$, the point at infinity determined by $\g{a}$ and the fact that $\alpha$ is a positive root.
If $A\cdot g(o)$, $g\in\Exp(\g{g}_\alpha\oplus\g{g}_{2\alpha})$, is another orbit of $A$, then it can be parametrized as $\beta(t)=\exp_{\gamma(t)}(r\xi_{\gamma(t)})$, where $r>0$ is a constant (the distance to $A\cdot o$), and $\xi$ is an equivariant normal vector field along $A\cdot o$. 
Now we apply the law of cosines~\cite[Corollary~1.4.4(3)]{Eberlein} to the points $o$, $\gamma(t)$ and $\beta(t)$.
Observe that $\lim_{t\to\infty}d(o,\gamma(t))=\infty$, but $d(\gamma(t),\beta(t))$ is bounded because $A\cdot o$ and $A\cdot g(o)$ are equidistant.
Hence, the angle $\sphericalangle_o(\gamma(t),\beta(t))$ subtended from $o$ between $\gamma(t)$ and $\beta(t)$ approaches $0$ as $t\to\infty$.
According to the definition of the cone topology of $\C H^n\cup\C H^n(\infty)$ (see for example~\cite[Proposition~1.7.6]{Eberlein}), we the conclude that $\lim_{t\to\infty}\beta(t)=\lim_{t\to\infty}\gamma(t)=x$.
An analogous argument shows that $\lim_{t\to-\infty}\beta(t)=\lim_{t\to-\infty}\gamma(t)=-x$, the other point at infinity of the geodesic $A\cdot o$.

Let $g_i=\Exp(2W_i+y_iZ)\in\Exp(\g{g}_\alpha\oplus\g{g}_{2\alpha})$, with $W_i\in\g{g}_\alpha$, $y_i\in\R$, $i\in\{1,2\}$. According to~\eqref{eq:Adg-ar} (with $T=0$, $\g{r}=0$), we have $\Ad(g_i^{-1})\g{h}=\R(B+W_i+y_iZ)$, $i\in\{1,2\}$.
Assume that there exists an isometry $\phi$ of the full isometry group of $\C H^n$ that maps $A\cdot g_1(o)$ to $A\cdot g_2(o)$.
Then, $\phi$ maps the limit points of one orbit to the limit points of the other.
Since these are $x$ and $-x$ by the discussion above, we conclude that $\phi$ leaves $\{x,-x\}\subset\C H^n(\infty)$ invariant.
In particular, the only geodesic of $\C H^n$ that has $\{x,-x\}$ as its limit set is $A\cdot o$.
Thus, $\phi$ maps $A\cdot o$~to~itself.

Hereafter $c$ denotes complex conjugation of projective coordinates of $\C H^n$ as a quotient of the pseudo-Hermitian flat space $\C^{1,n}\setminus\{0\}$.
Then, $c$ is an isometry of $\C H^n$ that is anti-holomorphic, but fixes $o$.
Considering the matrix expressions for~$\g{a}$, $\g{g}_\alpha$ and~$\g{g}_{2\alpha}$~\cite[\S3.1]{DKV}, it follows that $\Ad(c)(B)=B$, $\Ad(c)\g{g}_\alpha=\g{g}_\alpha$ and $\Ad(c)(Z)=-Z$.
In particular, $c$ fixes $x$.

There is an element $a\in A$ such that $\phi a(o)=o$. Hence, $k=\phi a$ maps $A\cdot o$ to itself, $A\cdot g_1(o)$ to $A\cdot g_2(o)$, and fixes $o$. Define $h=\sigma k$, where $\sigma$ is the identity transformation if $k(x)=x$, or the geodesic symmetry at $o$ if $k(x)=-x$. Then $h(x)=x$, which implies that $h\in \tilde{K}_0=K_0\sqcup cK_0$. Since $\sigma$ normalizes $A$, we have $h(A \cdot g_1(o))=\sigma k(A\cdot g_1(o))=\sigma (A\cdot g_2(o))=A\cdot \sigma(g_2(o))$. It is not difficult to check that there exists a unique $g\in N$ such that $g(o)\in A\cdot \sigma(g_2(o))$, and if $g=\Exp(2W+yZ)$, $W\in\g{g}_\alpha$, $y\in\R$, then $\|W\|=\|W_2\|$ and $|y|=|y_2|$.
As $\tilde{K}_0$ normalizes $AN$, we have $h_*\vert _{T_o\C H^n}\equiv\Ad(h)\vert_{\g{a}\oplus\g{n}}$.
Since $h(A\cdot g_1(o))=A\cdot g(o)$, and $\tilde{K}_0$ acts trivially on $\g{a}$ and leaves $\g{g}_\alpha$ and $\g{g}_{2\alpha}$ invariant, we have
\begin{align*}
\R(B+W+y Z)
&{}=\Ad(g^{-1})(\g{a})=\Ad(h)\Ad(g_1^{-1})(\g{a})\\
&{}=\Ad(h)(\R(B+W_1+y_1 Z))=\R(B+\Ad(h)W_1\pm y_1 Z).
\end{align*}
As $\tilde{K}_0$ acts transitively on the spheres of $\g{g}_\alpha$, we get $\lVert W_1\rVert=\|W\|=\lVert W_2\rVert$ and $| y_1|=|y|=|y_2|$.
This finishes the argument for $r=0$.

Now we show the converse.
The connected component of the identity of the normalizer of $\g{r}$ in $K_0$, which is given by $N_{K_0}^0(\g{r})\cong \SO(\dim\g{r})\times\U(n-1-\dim\g{r})$, acts transitively on the spheres of $\g{g}_\alpha\ominus\C\g{r}$ centered at the origin.
Thus, if $\lVert W_1\rVert=\lVert W_2\rVert$ and $y_1=y_2$, the orbits $H\cdot g_1(o)$ and $H\cdot g_2(o)$ are congruent.

We finally show that the congruence class does not depend on the sign of $y$.
We use the complex conjugation $c$ considered above.
We can find an element of $K_0$ that maps the totally real subspace $\g{r}$ to a subspace of $\g{g}_\alpha$ whose elements are real vectors; then we can assume $\Ad(c)\vert_\g{r}=\Id_{\g{r}}$.
Thus, supposing without loss of generality that $W\in\g{g}_\alpha\ominus\C\g{r}$ is real,  we have $\Ad(c)(B+W+yZ)=B+W-yZ$ and $\Ad(c)\g{r}=\g{r}$, as we wanted to show.

\subsection*{Case~(\ref{th:A:acrz})}

Let $\g{h}=\g{a}\oplus\g{c}\oplus\g{r}\oplus\g{g}_{2\alpha}$, with $\g{c}$ complex and $\g{r}$ totally real in $\g{g}_\alpha$.

If $H_{1}$ and $H_{2}$ are connected Lie subgroups of $G$ whose Lie algebras are $\g{h}_{i}=\g{a}\oplus\g{c}_{i}\oplus\g{r}_{i}\oplus\g{g}_{2\alpha}$, where $\g{c}_i$ is complex and $\g{r}_i$ is totally real in $\g{g}_\alpha$, $i\in\{1,2\}$,
then $\g{h}_1$ and $\g{h}_2$ are conjugate if and only if $\dim\g{c}_1=\dim\g{c}_2$ and $\dim\g{r}_1=\dim\g{r}_2$, because isometries of $\SU(1,n)$ are holomorphic.
Thus, from now on we fix $\g{c}$ and $\g{r}$.

Recall from Theorem~A that the CR $H$-orbits are obtained as $H\cdot g(o)$, with $g\in\Exp(J\g{r})$.

\begin{lemma}\label{lemma:shape:Adg-acrz}
The squared norm of the mean curvature vector of the orbit $H\cdot\Exp(JT)(o)$, $T\in\g{r}$, is given by
\[
\lVert\mean\rVert^2
=\frac{\lVert T\rVert^2\bigl(3+\dim(\g{c}\oplus\g{r})\bigr)^2}{4(4+\lVert T\rVert^2)}.
\]
\end{lemma}

\begin{proof}
By virtue of~\eqref{eq:Adg-acrz} with $W=0$ and~\cite[Lemma 6.1]{DDK}, the mean curvature of $H\cdot g(o)$~reads $\mean=
(3+\dim(\g{c}\oplus\g{r})) (\lVert T\rVert^2 B-2JT)/(2(4+\lVert T\rVert^{2}))$.
The formula in the statement follows after calculating the squared norm of this vector.
\end{proof}

Let $T_{1}$, $T_{2}\in\g{r}$, and define $g_{1}=\Exp(JT_{1})$, $g_{2}=\Exp(JT_{2})$.
We determine when the orbits $H\cdot g_1(o)$ and $H\cdot g_2(o)$ are congruent.

If $\lVert T_{1}\rVert=\lVert T_{2}\rVert$, then we show that $H\cdot g_1(o)$ and $H\cdot g_2(o)$ are congruent.
Recall from~\eqref{eq:Adg-acrz} (with $W=0$) that $\Ad(g_i^{-1})\g{h}=\R(2B+JT_i)\oplus\g{c}\oplus\g{r}\oplus\g{g}_{2\alpha}$.
The normalizer of $\g{r}$ on $K_{0}$,
$N_{K_{0}}^0(\g{r})\cong \SO(\dim\g{r})\times \U(n-1-\dim\g{r})$,
acts transitively on the spheres of $\g{r}$, and thus, there exists $k\in N_{K_0}^0(\g{r})$ such that $\Ad(k)JT_1=JT_2$ and $\Ad(k)\g{c}=\g{c}$.
Then, $\Ad(k)\Ad(g_1^{-1})\g{h}=\Ad(g_2^{-1})\g{h}$, and the orbits $H\cdot g_1(o)$ and $H\cdot g_2(o)$ are congruent.

Conversely, the function $h\colon[0,\infty)\to[0,\infty)$, $t\mapsto at/(4+t)$, $a>0$, satisfies $h'(t)=4a/(4+t)^2>0$. Hence, $h$ is injective, and Lemma~\ref{lemma:shape:Adg-acrz} implies that the orbits $H\cdot g_{1}(o)$ and $H\cdot g_{2}(o)$ are congruent if and only if $\lVert T_{1}\rVert=\lVert T_{2}\rVert$.

\subsection*{Non-congruence of the different types}\hfill

We finally study the congruence among the four different types of orbits listed in Theorem~A.
In order to do so, we first note that orbits of Type~(\ref{th:A:r}) and Type~(\ref{th:A:crz}) are contained in horospheres, while none of the orbits of Type~(\ref{th:A:ar}) or Type~(\ref{th:A:acrz}) satisfy this property. Considering this fact, it follows that none of the orbits of Type~(\ref{th:A:r}) or~(\ref{th:A:crz}) is congruent to any orbit of Type~(\ref{th:A:ar}) or~(\ref{th:A:acrz}).

On the other hand, every Type~(\ref{th:A:ar}) orbit is a totally real submanifold, while any orbit of Type~(\ref{th:A:acrz}) has non-trivial holomorphic part. Thus, none of the orbits of Type~(\ref{th:A:ar}) is congruent to any Type~(\ref{th:A:acrz}) orbit.

It only remains to analyze the congruence between orbits of Types (i) and (ii). Let us denote by $H_{i}$ the connected Lie subgroup of $G$ with Lie algebra $\g{h}_{i}$, $i\in\{1,2\}$, with $\g{h}_{1}=(\g{r}_{1}\ominus\R T)\oplus\R(T-\lVert T\rVert^2 Z)$, $T\in \g{r}_{1}$, and $\g{h}_{2}=\g{c}_{2}\oplus\g{r}_{2}\oplus\g{g}_{2\alpha}$.
As usual $\g{r}_{i}$ denotes a totally real subspace of $\g{g}_{\alpha}$ for each $i\in\{1,2\}$, and $\g{c}_{2}\subset\g{g}_{\alpha}$ denotes a complex one.

Suppose that an $H_1$-orbit is congruent to an $H_2$-orbit. Since $\g{h}_1$ is totally real, we must have $\g{c}_{2}=0$.
In this case we also have $r=\dim\g{r}_1=\dim\g{r}_2+1\geq 1$.
Moreover, we must have $\lVert\mean_1\rVert^2=\lVert\mean_2\rVert^2$, which according to Lemmas~\ref{lemma:shape:Adg-r} and~\ref{lemma:shape:Adg-crz}, implies
\[
\frac{(1+r)^2}{4}=\frac{4\lVert T\rVert^2+(r+(r+1)\lVert T\rVert^2)^2}{4(1+\lVert T\rVert^2)^2},
\]
or equivalently, $3+2(r-1)(1+\lVert T\rVert^2)=0$.
Since this never happens, none of the orbits of $H_1$ is congruent to any orbit of $H_2$.


\end{document}